\documentclass[10pt,reqno]{amsart}
\usepackage{latexsym,amssymb,amscd}
\usepackage{amsmath}
\def\NZQ{\Bbb}               

\def\ZZ{{\NZQ Z}}

\def\B'c{{\mathcal{B'}}}
\def\U'c{{\mathcal{U'}}}

\def\opn#1#2{\def#1{\operatorname{#2}}} 
\opn\chara{char}
\opn\length{\ell}
\opn\projdim{proj\,dim}
\opn\injdim{inj\,dim}
\opn\ini{in}
\opn\rank{rank}
\opn\depth{depth}
\opn\sdepth{sdepth}
\opn\height{ht}
\opn\embdim{emb\,dim}
\opn\codim{codim}

\opn\Tr{Tr}
\opn\bigrank{big\,rank}
\opn\superheight{superheight}\opn\lcm{lcm}
\opn\trdeg{tr\,deg}%
\opn\reg{reg}
\opn\lreg{lreg}
\opn\set{set}
\opn\supp{Supp}
\opn\shad{Shad}
\opn\div{div}
\opn\Div{Div}
\opn\cl{cl}
\opn\Cl{Cl}
\opn\Spec{Spec}
\opn\Supp{Supp}
\opn\supp{supp}
\opn\Sing{Sing}
\opn\Ass{Ass}
\opn\Min{Min}
\opn\size{size}
\opn\bigsize{bigsize}
\opn\lex{lex}
\opn\Ann{Ann}
\opn\Rad{Rad}
\opn\Soc{Soc}
\opn\Ker{Ker}
\opn\Coker{Coker}
\opn\Im{Im}
\opn\Hom{Hom}
\opn\Tor{Tor}
\opn\Ext{Ext}
\opn\End{End}
\opn\Aut{Aut}
\opn\id{id}

\opn\nat{nat}
\opn\GL{GL}
\opn\SL{SL}
\opn\mod{mod}
\opn\ord{ord}
\opn\aff{aff}
\opn\con{conv}
\opn\relint{relint}
\opn\st{st}
\opn\lk{lk}
\opn\cn{cn}
\opn\core{core}
\opn\vol{vol}
\opn\gr{gr}

\def\pot#1#2{#1[\kern-0.28ex[#2]\kern-0.28ex]}

\opn\dirlim{\underrightarrow{\lim}}
\opn\invlim{\underleftarrow{\lim}}
\def\pnt{{\raise0.5mm\hbox{\large\bf.}}}

\def\Implies{\ifmmode\Longrightarrow \else
     \unskip${}\Longrightarrow{}$\ignorespaces\fi}
\def\implies{\ifmmode\Rightarrow \else
     \unskip${}\Rightarrow{}$\ignorespaces\fi}
\def\iff{\ifmmode\Longleftrightarrow \else
     \unskip${}\Longleftrightarrow{}$\ignorespaces\fi}

\let\:=\colon
\newtheorem{Theorem}{Theorem}[section]
\newtheorem{Lemma}[Theorem]{Lemma}
\newtheorem{Corollary}[Theorem]{Corollary}
\newtheorem{Proposition}[Theorem]{Proposition}
\newtheorem{Remark}[Theorem]{Remark}

\newtheorem{Example}[Theorem]{Example}

\newtheorem{Definition}[Theorem]{Definition}

\newtheorem{Conjecture}[Theorem]{Conjecture}

\let\epsilon=\varepsilon
\let\phi=\varphi
\let\kappa=\varkappa
\numberwithin{equation}{section}

\title{Depth and Stanley depth of the edge ideals of the powers of paths and cycles}
\author[Zahid Iqbal]{Zahid Iqbal}
\address{Zahid Iqbal, School of Natural Sciences, National University of Sciences and Technology Islamabad, Sector H-12, Islamabad Pakistan.}
\email{786zahidwarraich@gmail.com}
\author[Muhammad Ishaq]{Muhammad Ishaq}
\address{Muhammad Ishaq, School of Natural Sciences, National University of Sciences and Technology Islamabad, Sector H-12, Islamabad Pakistan.}
\email{ishaq$\_\,$maths@yahoo.com}

\begin{document}
\maketitle
\begin{abstract}
Let $k$ be a positive integer. We compute depth and Stanley depth of the quotient ring of the edge ideal associated to the $k^{th}$ power of a path on $n$ vertices. We show that both depth and Stanley depth have the same values and can be given in terms of $k$ and $n$. For $n\geq 2k+1$, let $n\equiv 0,k+1,k+2,\dots, n-1(\mod(2k+1))$. Then we give values of depth and Stanley depth of the quotient ring of the edge ideal associated to the $k^{th}$ power of a cycle on $n$ vertices and tight bounds otherwise, in terms of $n$ and $k$. We also compute lower bounds for the Stanley depth of the edge ideals associated to the $k^{th}$ power of a path and a cycle and prove a conjecture of Herzog presented in \cite{HS} for these ideals.
\end{abstract}
\section{Introduction}
\noindent Let $K$ be a field and $S := K[x_{1}, \ldots, x_{n}]$ the
polynomial ring over $K$. Let $M$ be a finitely generated
$\mathbb{Z}^{n}$-graded $S$-module. A Stanley decomposition of $M$ is a
presentation of $K$-vector space $M$ as a finite direct sum
$$\mathcal{D}: M = \bigoplus_{i = 1}^{s}v_{i}K[W_{i}],$$ where $v_{i}\in M$, $W_{i}\in \{x_{1}, \ldots, x_{n}\}$ such that
$v_{i}K[W_{i}]$ denotes the $K$-subspace of $M$, which is generated
by all elements $v_iw$, where $w$ is a monomial in $K[W_{i}]$. The
$\mathbb{Z}^{n}$-graded $K$-subspace $v_{i}K[W_{i}]\subset M$ is
called a Stanley space of dimension $|W_{i}|$, if $v_{i}K[W_{i}]$ is
a free $K[W_{i}]$-module, where $|W_{i}|$ denotes the cardinality of $W_i$. Define
$$\sdepth(\mathcal{D}) = \min\{|W_{i}|: i = 1, \ldots, s\},$$ and
$$\sdepth(M) = \max\{\sdepth(\mathcal{D}): \text{$\mathcal{D}$~is~a~Stanley~decomposition~of
~$M$}\}.$$ The number $\sdepth(\mathcal{D})$ is called the Stanley
depth of decomposition $\mathcal{D}$ and $\sdepth(M)$ is called the Stanley depth
of $M$. Stanley conjectured in \cite{RP} that $\sdepth(M)\geq \depth(M)$ for any $\ZZ^n$-graded $S$-module $M$. This conjecture was disproved by Duval et al. in \cite{DG} as was expected due to different natures of both the invariants. However, the relation between Stanley depth and some other invariants has already been established see \cite{HVZ,HDV,D3,ZT}. In \cite{HVZ}, Herzog, Vladoiu and Zheng proved that the Stanley depth of $M$ can be computed in a finite number of steps, if $M = J/I$, where $I\subset J\subset S$ are monomial ideals. But practically it is too hard to compute Stanley depth by using this method see for instance, \cite{BH,MC8,KS,KY}. For computing Stanley depth for some classes of modules we refer the reader to \cite{MI,ID,PFY,AR1}. In this paper we attempt to find values and reasonable bounds for depth and Stanley depth of $I$ and $S/I$, where $I$ is the edge ideal of a power of a path or a cycle. We also compare the values of $\sdepth(I)$ and $\sdepth(S/I)$ and give positive answers to the following conjecture of Herzog.
\begin{Conjecture}\cite{HS}\label{C1}
Let $I\subset S$ be a monomial ideal then $$\sdepth(I)\geq \sdepth(S/I).$$
\end{Conjecture}
The above conjecture has been proved in some special cases see \cite{ZIA,KY,ID,AR1}. The paper is organized as follows: First two sections are devoted to introduction, definitions, notations, and discussion of some known results. In third section, we compute depth and Stanley depth of $S/I(P_n^k)$ as Theorems (\ref{Th1} and \ref{Th11}), where $I(P_n^k)$ denotes the edge ideal of the $k^{th}$ power of a path $P_n$ on $n$ vertices. These theorems prove that
$$\depth(S/I(P_n^k))=\sdepth(S/I(P_n^k))=\lceil\frac{n}{2k+1}\rceil.$$
Let $I(C_n^k)$ be the edge ideal of the $k^{th}$ power of a cycle $C_n$ on $n$ vertices. In fourth section we give some lower bounds for depth and Stanley depth of $S/I(C_n^k)$ as Theorems (\ref{Th22} and \ref{Th33}). For $n\geq 2k+1$, our Corollaries (\ref{Cor11} and \ref{Cor22}) show that
\begin{eqnarray*}
&\depth(S/I(C_n^k))=\sdepth(S/I(C_n^k))=\lceil\frac{n}{2k+1}\rceil, \text{\,\, if\,\, $n\equiv 0,k+1,\dots,n-1(\mod(2k+1))$};&\\
&\lceil\frac{n}{2k+1}\rceil-1\leq \depth(S/I(C_n^k)),\sdepth(S/I(C_n^k))\leq \lceil\frac{n}{2k+1}\rceil, \text{ if $n\equiv 1,2,\dots,k(\mod(2k+1)).$}&
\end{eqnarray*}
Last section is devoted to Conjecture \ref{C1} for $I(P_n^k)$ and $I(C_n^k)$. By our Theorem \ref{th3} we have $$\sdepth(I(P_n^k))\geq \lceil\frac{n}{2k+1}\rceil+1,$$ which shows that $I(P_n^k)$ satisfies Conjecture \ref{C1}. Let $n\geq 2k+1$,
Proposition \ref{pr3} gives a lower bound for $I(C_{n}^k)/I(P_{n}^k)$ that is  $$\sdepth({I(C_{n}^k)}/{I(P_{n}^k)})\geq\lceil\frac{n+k+1}{2k+1}\rceil.$$
Corollary \ref{Cor4} of this paper proves that $I(C_n^k)$ satisfies Conjecture \ref{C1}. The above results are proved in \cite{ZIA} when $k=2$.
\section{Definitions and notation}
Throughout this paper, $\mathfrak{m}$ denotes the unique maximal graded ideal $(x_1,x_2,\dots,x_n)$ of $S$. We set $S_m:=K[x_1,x_2,\dots,x_m]$, $\supp(v):=\{i:x_i|v\}$ and $\supp(I):=\{i:x_i|u, \text{ for some } u\in \mathcal{G}(I)\}$, where $\mathcal{G}(I)$ denotes the unique minimal set of monomial generators of the monomial ideal $I$.  Let $I\subset S$ be an ideal then we write $I$ instead of $IS$. Thus every ideal will be considered an ideal of $S$ unless otherwise stated. Let $I$ and $J$ be monomial ideals of $S$, then for $I+J$ we write $(I,J)$.
\begin{Definition}
{\em
Let $G$ be a simple graph, for a positive integer $k$ the $k^{th}$ power of graph $G$ is another graph $G^k$ on the same set of vertices, such that two vertices are adjacent in $G^k$ when their distance in $G$ is at most $k$.
}
\end{Definition}
In the whole paper we label the vertices of the graph $G$ by $1,2,\dots,n$. We denote the set of vertices of G by $[n]:=\{1,2,\dots,n\}$. We assume that all graphs and their powers are simple graphs. We also assume that all graphs have at least two vertices and a non-empty edge set.
\begin{Definition}
{\em
Let $G$ be a graph then the edge ideal $I(G)$ associated to $G$ is the square free monomial ideal of $S$ that is
\begin{equation*}
I(G)=(x_{i}x_{j} : \{i, j\}\in E(G)).
\end{equation*}
}
\end{Definition}
\begin{Definition}
{\em
Let $G$ be a graph, then $G$ is called a path if $E(G)=\{\{i,i+1\}:i\in[n-1]\}$. A path on $n$ vertices is usually denoted by $P_n$.
}
\end{Definition}

\begin{Definition}
{\em
Let $n\geq3$, a graph $G$ is called a cycle if $E(G)=\{\{i,i+1\}:i\in[n-1]\}\cup\{1,n\}$. A cycle on $n$ vertices is usually denoted by $C_n$.
}
\end{Definition}
\begin{Definition}
{\em
For $n\geq 2$, the $k^{th}$ power of a path, denoted by $P_n^k$, is a graph such that for all $1\leq i<j\leq n$, $\{i,j\}\in E(P^{k}_{n})$ if and only if $0<j-i\leq k$. That is the edge set of $P_n^k$ is:
\begin{description}
\item[(1)] If $n\leq k+1$, then $E(P_{n}^k)=\big\{\{i,j\}:1\leq i<j\leq n\big\}$.
\item[(2)] If $n\geq k+2$, then
\begin{eqnarray*}
E(P_n^k)=\bigcup_{i=1}^{n-k}\big\{\{i,i+1\},\{i,i+2\},\dots,\{i,i+k\}\big\}\bigcup \bigcup_{j=n-k+1}^{n-1}\big\{\{j,j+1\},\{j,j+2\},\dots,\{j,n\}\big\}.
\end{eqnarray*}
\end{description}
}

\end{Definition}
\begin{Definition}
{\em For $n\geq 3$, the $k^{th}$ power of a cycle, denoted by $C_n^k$, is a graph such that for all vertices $1\leq i,j\leq n$, $\{i,j\}\in E(C^{k}_{n})$ if and only if $|j-i|\leq k$ or $|j-i|\geq n-k$. The edge set of $C_n^k$ is:
\begin{description}
\item[(1)] If $n\leq 2k+1$, then $E(C_{n}^k)=\big\{\{i,j\}:1\leq i<j\leq n\big\}$.
\item[(2)] If $n\geq 2k+2$, then
\begin{eqnarray*}
E(C_n^k)= E(P_n^k)\bigcup\bigcup_{l=1}^{k}\big\{\{l,l+n-k\},\{l,l+n-k+1\},\{l,l+n-k+2\}\dots,\{l,n\}\big\}.
\end{eqnarray*}
\end{description}
}
\end{Definition}

The minimal generating set for the edge ideal of $P_n^k$ is:
\begin{description}
\item[(1)] If $n\leq k+1$, then $\mathcal{G}(I(P_{n}^k))=\{x_ix_{j}:1\leq i<j\leq n\}$. That is $I(P_{n}^k)$ is a square free Veronese ideal of degree 2 in variables $x_1,x_2,\dots,x_n$.
\item[(2)] If $n\geq k+2$, then
$$\mathcal{G}(I(P_n^k))=\bigcup_{i=1}^{n-k}\{x_ix_{i+1},x_ix_{i+2},\dots,x_ix_{i+k}\} \bigcup\bigcup_{j=n-k+1}^{n-1}\{x_j x_{j+1},x_j x_{j+2},\dots,x_j x_n\}.$$
\end{description}

The minimal generating set for the edge ideal of $C_n^k$ is:
\begin{description}
\item[(1)] If $n\leq 2k+1$, then $\mathcal{G}(I(C_n^k))=\{x_{i}x_{j}:1\leq i<j\leq n\}$. That is $I(C_{n}^k)$ is a square free Veronese ideal of degree 2 in variables $x_1,x_2,\dots,x_n$.
\item[(2)] If $n\geq 2k+2$, then
$$\mathcal{G}(I(C_n^k))=\mathcal{G}(I(P_n^k))\bigcup\bigcup_{l=1}^{k}\{x_l x_{l+n-k},x_l x_{l+n-k+1},\dots,x_l x_n\}.$$
\end{description}
\begin{Lemma}\cite[Lemma 3]{Lin}
{\em  For $n\ge k+1$, $|\mathcal{G}(I(P_n^k))|=nk-\frac{k(k+1)}{2}$.}
\end{Lemma}
\begin{Remark}
{\em If $n\geq 2k+1$, then $|\mathcal{G}(I(C_n^k))|=nk$.}
\end{Remark}
\begin{Definition}
{\em Let $G$ be a graph and let $i\in [n]$, then we define $$N_G(x_i):=\{x_j:x_ix_j\in \mathcal{G}\big(I(G)\big)\},$$ where $j\in [n]\backslash \{i\}$}.
\end{Definition}

Let $k\geq 2$, $0\leq i\leq k-1$ and $n\geq 2k+2$, then in the following we introduce some particular monomial prime ideals of $S$. $$A_{n-k+i}:=(x_{n-k},x_{n-k+1},\dots,x_{n-k+i}).$$
\big(We set $A_{n-k-1}=(0)$\big).
$$B_{n-k+i}:=\big(N_{P_n^k}(x_{n-k+i})\big)=(x_{n-2k+i},x_{n-2k+i+1},\dots,x_{n-k+i-1},x_{n-k+i+1},\dots,x_n).$$
and
\begin{multline*}
D_{n-k+i}:=\big(N_{C_n^k}(x_{n-k+i})\big)\\=\left\{
               \begin{array}{ll}
                 (x_{n-2k},x_{n-2k+1},\dots,x_{n-k-1},x_{n-k+1},\dots,x_n), & \hbox{if $i=0$;} \\
                 (x_{n-2k+i},x_{n-2k+i+1},\dots,x_{n-k+i-1},x_{n-k+i+1},\dots,x_n,x_1,x_2,\dots,x_{i}), & \hbox{if $1\leq i\leq k-1$.}
               \end{array}
             \right.
\end{multline*}
These monomial prime ideals and the following function play important role in the proof of our main theorems. Let $k\geq 2$, $0\leq i\leq k-1$ and $2k+2\leq n\leq 3k+1$, we define a function $$f:\{n-k,n-k+1,\dots,n-k+i,\dots,n-1\}\longrightarrow \ZZ^+\cup \{0\},\text{\,\,by}$$
$$f(n-k+i)=\left\{
         \begin{array}{ll}
           k, & \hbox{if\, $n-2k-1+i\geq k+1$;} \\
           n-2k-2+i, & \hbox{if\, $2\leq n-2k-1+i<k+1$.} \\
                    \end{array}
       \right.
$$
In the following we recall some known results that we refer several times in this paper.
\begin{Lemma}\label{le01}
(Depth Lemma) If $0\longrightarrow U \longrightarrow M \longrightarrow N \longrightarrow 0$ is a short exact sequence of modules over a
local ring $S$, or a Noetherian graded ring with $S_0$ local, then
\begin{enumerate}
\item $\depth M \geq \min\{\depth N, \depth U\}$.
\item $\depth U \geq \min\{\depth M, \depth N + 1\}$.
\item $\depth N \geq \min\{\depth U- 1, \depth M\}$.
\end{enumerate}
\end{Lemma}
\begin{Lemma}\label{le02}\cite{AR1}
Let $0\longrightarrow U\longrightarrow V\longrightarrow W\longrightarrow 0$ be a short exact sequence of $\ZZ^{n}$-graded $S$-modules. Then
\begin{equation*}
\sdepth(V)\geq\min\{\sdepth(U),\sdepth(W)\}.
\end{equation*}
\end{Lemma}
\begin{Lemma}\cite{ZIA}\label{zia}
Let $I\subset S$ be a squarefree monomial ideal with $\supp(I)=[n]$, let $v:=x_{i_1}x_{i_2}\cdots x_{i_q}\in S/I$, such that $x_jv\in I$, for all $j\in [n]\backslash \supp(v)$. Then $\sdepth(S/I)\leq q$.
\end{Lemma}

\section{depth and Stanley of cyclic modules associated to the edge ideals of the powers of a path}
In this section, we compute depth and Stanley depth of $S/I(P_n^k)$. We start this section with the following elementary lemma:
\begin{Lemma}\label{led}
Let $a\geq 2$ be an integer, $\{E_i:1\leq i\leq a\}$ and $\{G_i:0\leq i\leq a\}$ be sequences of $\ZZ^n$-graged $S$ modules, and we have the following short exact sequences:
\begin{equation}\tag{1}\label{eq1}
0\longrightarrow E_1\longrightarrow G_0\longrightarrow G_1\longrightarrow 0
\end{equation}
\begin{equation}\tag{2}\label{eq2}
0\longrightarrow E_2\longrightarrow G_1\longrightarrow G_2\longrightarrow 0
\end{equation}
$$\vdots$$
\begin{equation}\tag{$a-1$} \label{eq$a-1$}
0\longrightarrow E_{a-1}\longrightarrow G_{a-2}\longrightarrow G_{a-1}\longrightarrow 0
\end{equation}
\begin{equation}\tag{$a$}\label{eq$a$}
0\longrightarrow E_a\longrightarrow G_{a-1}\longrightarrow G_a\longrightarrow 0
\end{equation}
such that $\depth(G_a)\geq \depth(E_a)$ and $\depth(E_{i})\geq \depth(E_{i-1})$, for all $2\leq i\leq a$, then $\depth(G_0)=\depth(E_1)$.
\end{Lemma}
\begin{proof}
By assumption $\depth(G_a)\geq \depth(E_a)$, applying Depth Lemma on the exact sequence (\ref{eq$a$}) we get $\depth(G_{a-1})=\depth(E_a)$. We also have by assumption $$\depth(G_{a-1})=\depth(E_a)\geq \depth(E_{a-1}),$$ applying Depth Lemma to the exact sequence (\ref{eq$a-1$}) we have $\depth(G_{a-2})=\depth(E_{a-1})$. We repeat the same steps on all exact sequences one by one from bottom to top and we get $\depth(G_{i-1})=\depth(E_i)$ for all $i$. Thus if $i=1$ then we have $\depth(G_0)=\depth(E_1)$.
\end{proof}
\begin{Lemma}\label{le1}
Let $k\geq 2$ and $n\geq 2k+2$, then $S/(I(P_n^k),A_{n-1})\cong S_{n-k-1}/I(P_{n-k-1}^{k})[x_n]$.
\end{Lemma}
\begin{proof}
We have $$\mathcal{G}(I(P_n^k))=\bigcup_{i=1}^{n-k}\big\{x_ix_{i+1},x_ix_{i+2},\dots,x_ix_{i+k}\big\} \bigcup\bigcup_{i=n-k+1}^{n-1}\big\{x_i x_{i+1},x_i x_{i+2},\dots,x_i x_n\big\}.$$
Then
\begin{multline*}
I(P_n^k)+A_{n-1}=\Big[\sum_{i=1}^{n-2k-1}(x_ix_{i+1},x_ix_{i+2},\dots,x_ix_{i+k}) +\sum_{i=n-2k}^{n-k}(x_ix_{i+1},x_ix_{i+2},\dots,x_ix_{i+k})+ \\ \sum_{i=n-k+1}^{n-1}(x_i x_{i+1},x_i x_{i+2},\dots,x_ix_n)\Big]+ A_{n-1}\\=\sum_{i=1}^{n-2k-1}\big(x_ix_{i+1},x_ix_{i+2},\dots,x_ix_{i+k}\big)+\\\sum_{i=n-2k}^{n-k-2}(x_i x_{i+1},x_i x_{i+2},\dots,x_i x_{n-k-1})+ A_{n-1}=I(P_{n-k-1}^k)+A_{n-1}.
\end{multline*}
Thus the required result follows.
\end{proof}
%
\begin{Lemma}\label{le2}
Let $k\geq 2$, $0\leq i\leq k-1$ and $n\geq 3k+2$, then $$S/(I(P_n^k):x_{n-k+i})\cong S_{n-2k-1+i}/I(P_{n-2k-1+i}^k)[x_{n-k+i}].$$
\end{Lemma}
\begin{proof}
It is enough to prove that $(I(P_n^k):x_{n-k+i})=(I(P_{n-2k-1+i}^k),B_{n-k+i})$. Clearly $$I(P_{n-2k-1+i}^k)\subset I(P_{n}^k)\subset (I(P_{n}^k):x_{n-k+i}).$$ Let $u\in B_{n-k+i}$, then by definition of $I(P_n^k)$, $ux_{n-k+i}\in I(P_n^k)$ that is $u\in (I(P_n^k):x_{n-k+i})$. Thus $B_{n-k+i}\subset (I(P_n^k):x_{n-k+i})$ and we have $\big(I(P_{n-2k-1+i}^k),B_{n-k+i}\big)\subset (I(P_n^k):x_{n-k+i})$. Now for the other inclusion, let $w$ be a monomial generator of $(I(P_n^k):x_{n-k+i})$, then $w=\frac{v}{gcd(v,x_{n-k+i})}$, where $v\in\mathcal{G}(I(P_n^k))$. If $\supp(v)\cap \mathcal{G}(B_{n-k+i})\neq\emptyset$, then we have $w\in \mathcal{G}(B_{n-k+i})$ and if $\supp(v)\cap \mathcal{G}(B_{n-k+i})=\emptyset$ then $$w\in \mathcal{G}(I(P_{n}^k))\cap K[x_1,x_2,\dots,x_{n-2k-1+i}]=\mathcal{G}(I(P_{n-2k-1+i}^k)).$$
\end{proof}
\begin{Lemma}\label{cor1}
Let $n\geq 3k+2$ and $0\leq i\leq k-1$, then we have $$S/((I(P_{n}^k),A_{n-k+(i-1)}):x_{n-k+i})\cong S_{n-2k-1+i}/I(P_{n-2k-1+i}^k)[x_{n-k+i}].$$
\end{Lemma}
\begin{proof}
Clearly $\big((I(P_{n}^k),A_{n-k+(i-1)}):x_{n-k+i}\big)=\big((I(P_{n}^k):x_{n-k+i}),A_{n-k+(i-1)}\big)$. Now using the proof of Lemma \ref{le2}
$$\big((I(P_{n}^k):x_{n-k+i}),A_{n-k+(i-1)}\big)=\big(I(P_{n-2k-1+i}^k),B_{n-k+i},A_{n-k+(i-1)}\big)=
\big(I(P_{n-2k-1+i}^k),B_{n-k+i}\big)$$
as $A_{n-k+(i-1)}\subset B_{n-k+i}$. Thus the required result follows by Lemma \ref{le2}.
\end{proof}
\begin{Remark}\label{rev}
{\em Let $m\geq 2$ and $I(P_{m}^{m-1})\subset S_m=K[x_1,x_2,\dots,x_m]$ be the edge ideal of the $(m-1)^{th}$ power of path $P_{m}$, then $I(P_{m}^{m-1})$ is a squarefree Veronese ideal of degree 2 in variables $x_1,x_2,\dots,x_m$. Thus by \cite[Corollary 10.3.7]{HH} and \cite[Theorem 1.1]{MC8} $$\depth(S_m/I(P_{m}^{m-1}))=\sdepth(S_m/I(P_{m}^{m-1}))=1.$$}
\end{Remark}
\begin{Remark}\label{re1}
{\em
Let $k\geq 2$ and $2k+2\leq n\leq 3k+1$, then it is easy to see that
\begin{description}
\item [(1)] If $n=2k+2$, then $S/(I(P_n^k):x_{n-k})=S/(x_2,\dots,x_{n-k-1},x_{n-k+1},\dots,x_n)\cong K[x_1,x_{n-k}].$
\item [(2)] If $0\leq i\leq k-1$ and $n>2k+2$, then
\begin{multline*}
$$S/(I(P_n^k):x_{n-k+i})=S/((I(P_n^k),A_{n-k+(i-1)}):x_{n-k+i})\\\cong S_{n-2k-1+i}/I(P_{n-2k-1+i}^{f(n-k+i)})=\left\{
                                                    \begin{array}{ll}
                                                      S_{n-2k-1+i}/I(P_{n-2k-1+i}^k), & \hbox{if $n-2k-1+i\geq k+1$;} \\\\
                                                      S_{n-2k-1+i}/I(P_{n-2k-1+i}^{n-2k-2+i}), & \hbox{otherwise.}
                                                    \end{array}
                                                  \right.$$
\end{multline*}
\end{description}
}
\end{Remark}
\begin{Theorem}\label{Th1}
Let $n\geq 2$, then $\depth(S/I(P_n^k))=\lceil\frac{n}{2k+1}\rceil$.
\end{Theorem}
\begin{proof}
\begin{description}
\item[(a)] Let $n\leq k+1$, then $I(P_n^k)$ is a square free Veronese ideal thus by Remark \ref{rev}, $\depth(S/I(P_n^k))=1=\lceil\frac{n}{2k+1}\rceil$.
\item[(b)] Let $n\geq k+2$, we consider the following cases:
\begin{description}
\item[(1)] If $k=1$, then by \cite[Lemma 2.8]{SM1} we have $\depth(S/I(P_n^1))=\lceil\frac{n}{3}\rceil=\lceil\frac{n}{2k+1}\rceil$.
\item[(2)] If $k\geq 2$, $k+2\leq n\leq 2k+1$, then we have $\depth(S/I(P_n^k))\geq 1$ as $\mathfrak{m}\notin \Ass(S/I(P_n^k))$. Since $x_{k+1}\notin I(P_n^k)$ and $x_sx_{k+1}\in \mathcal{G}(I(P_n^k))$ for all $s\in \{1,\dots k,k+2,\dots,n\}$, therefore $$(I(P_n^k):x_{k+1})=(x_1,\dots x_k,x_{k+2},\dots,x_n).$$
By \cite[Corollary 1.3]{AR1}, we have
$$\depth(S/I(P_n^k))\leq \depth(S/(I(P_n^k):x_{k+1}))=\depth(S/(x_1,\dots x_k,x_{k+2},\dots,x_n))=1.$$
Thus $\depth(S/I(P_n^k))=1=\lceil\frac{n}{2k+1}\rceil$.
\item[(3)] If $k\geq 2$, $2k+2\leq n\leq 3k+1$, now consider the following short exact sequences:
$$0\longrightarrow S/(I(P_n^k):x_{n-k})\xrightarrow[]{\cdot x_{n-k}} S/I(P_n^k)\longrightarrow S/(I(P_n^k),x_{n-k})\longrightarrow0$$
$$0\longrightarrow S/((I(P_n^k),A_{n-k}):x_{n-k+1})\xrightarrow[]{\cdot x_{n-k+1}} S/(I(P_n^k),A_{n-k})\longrightarrow S/(I(P_n^k),A_{n-k+1})\longrightarrow0$$
$$\vdots$$
$$0\longrightarrow S/((I(P_n^k),A_{n-k+(i-1)}):x_{n-k+i})\xrightarrow[]{\cdot x_{n-k+i}} S/(I(P_n^k),A_{n-k+(i-1)})\longrightarrow S/(I(P_n^k),A_{n-k+i})\longrightarrow0$$
$$\vdots$$
$$0\longrightarrow S/((I(P_n^k),A_{n-2}):x_{n-1})\xrightarrow[]{\cdot x_{n-1}} S/(I(P_n^k),A_{n-2})\longrightarrow S/(I(P_n^k),A_{n-1})\longrightarrow0.$$\\
By Lemma \ref{le1}, $S/(I(P_n^k),A_{n-1})\cong S_{n-k-1}/I(P_{n-k-1}^k)[x_{n}]$. Since $k+1\leq n-k-1\leq 2k$, thus if $n-k-1=k+1$ then by case(a) and \cite[Lemma 3.6]{HVZ}, otherwise by case(b)(2) and \cite[Lemma 3.6]{HVZ} we have $$\depth\big(S/(I(P_n^k),A_{n-1})\big)=2.$$
Now we show that $\depth\big(S/(I(P_n^k):x_{n-k})\big)=2$, for this we consider two cases:\\
If $n=2k+2$, then by Remark \ref{re1}
$$S/(I(P_n^k):x_{n-k})=S/(x_2,x_3,\dots,x_{n-k-1},x_{n-k+1},\dots,x_n)\cong K[x_1,x_{n-k}],$$
and thus $\depth\big(S/(I(P_n^k):x_{n-k})\big)=2$.\\
If $n>2k+2$, by Remark \ref{re1} we have
$$S/(I(P_n^k):x_{n-k})\cong S_{n-2k-1}/I(P_{n-2k-1}^{n-2k-2})[x_{n-k}],$$
where $2\leq n-2k-1\leq k$. Thus by Remark \ref{rev} and \cite[Lemma 3.6]{HVZ} we get $$\depth\big(S/(I(P_n^k):x_{n-k})\big)=2.$$
Now let $1\leq i\leq k-1$, by Remark \ref{re1}
\begin{multline*}
$$S/((I(P_n^k),A_{n-k+(i-1)}):x_{n-k+i})=S/(I(P_n^k):x_{n-k+i})\cong S_{n-2k-1+i}/I(P_{n-2k-1+i}^{f(n-k+i)})[x_{n-k+i}].$$
\end{multline*}
Let $T:=S_{n-2k-1+i}/I(P_{n-2k-1+i}^{f(n-k+i)})[x_{n-k+i}]$, we consider the following cases:
\begin{description}
\item [(i)] If $k+1=n-2k-1+i$, then $T=S_{k+1}/I(P_{k+1}^k)[x_{n-k+i}]$, thus by case(a) and \cite[Lemma 3.6]{HVZ} we have $\depth(T)=2$.\\
\item [(ii)] If $k+1<n-2k-1+i$, then $T=S_{n-2k-1+i}/I(P_{n-2k-1+i}^k)[x_{n-k+i}]$. Since $k+2\leq n-2k-1+i\leq 2k-1$, thus by case(b)(2) and \cite[Lemma 3.6]{HVZ} we have $\depth(T)=2$.\\
\item [(iii)] If $2\leq n-2k-1+i< k+1$, then $T=S_{n-2k-1+i}/I(P_{n-2k-1+i}^{n-2k-2+i})[x_{n-k+i}]$, by Remark \ref{rev} and \cite[Lemma 3.6]{HVZ} we have $\depth(T)=2$.\\
\end{description}
Thus by Lemma \ref{led} we have $\depth(S/I(P_n^k))=2$.\\
\item[(4)] If $k\geq 2$, $n\geq 3k+2$, then we consider the family of short exact sequences:
  $$0\longrightarrow S/(I(P_n^k):x_{n-k})=S/((I(P_n^k),A_{n-k-1}):x_{n-k}))\xrightarrow[]{\cdot x_{n-k}} S/I(P_n^k)\longrightarrow S/(I(P_n^k),A_{n-k})\longrightarrow 0$$
  $$0\longrightarrow S/((I(P_n^k),A_{n-k}):x_{n-k+1})\xrightarrow[]{\cdot x_{n-k+1}} S/(I(P_n^k),A_{n-k})\longrightarrow S/(I(P_n^k),A_{n-k+1})\longrightarrow 0$$
   $$\vdots$$
   $$0\longrightarrow S/((I(P_n^k),A_{n-k+(i-1)}):x_{n-k+i})\xrightarrow[]{\cdot x_{n-k+i}} S/(I(P_n^k),A_{n-k+(i-1)})\longrightarrow S/(I(P_n^k),A_{n-k+i})\longrightarrow 0$$
   $$\vdots$$
   $$0\longrightarrow S/((I(P_n^k),A_{n-2}):x_{n-1})\xrightarrow[]{\cdot x_{n-1}} S/(I(P_n^k),A_{n-2})\longrightarrow S/(I(P_n^k),A_{n-1})\longrightarrow 0.$$\\
By Lemma \ref{le1}, $S/(I(P_n^k),A_{n-1}))\cong S_{n-k-1}/I(P_{n-k-1}^k)[x_n].$
Thus by induction on $n$ and \cite[Lemma 3.6]{HVZ} we have
$$\depth (S/(I(P_n^k),A_{n-1}))=\lceil\frac{n-k-1}{2k+1}\rceil+1.$$
Let $0\leq i\leq k-1$, by Lemma \ref{cor1} we have
$$S/((I(P_n^k),A_{n-k+(i-1)}):x_{n-k+i})\cong S_{n-2k-1+i}/I(P_{n-2k-1+i}^k)[x_{n-k+i}].$$
Thus by induction on $n$ and \cite[Lemma 3.6]{HVZ} we have $$\depth(S/((I(P_n^k),A_{n-k+(i-1)}):x_{n-k+i}))=\lceil\frac{n-2k-1+i}{2k+1}\rceil+1.$$
We can see that $$\depth (S/(I(P_n^k),A_{n-1}))=\lceil\frac{n-k-1}{2k+1}\rceil+1\geq \lceil\frac{n-k-2}{2k+1}\rceil+1=\depth (S/(I(P_n^k),A_{n-2}):x_{n-1})).$$
And for all $1\leq i\leq k-1$,
\begin{multline*}
\depth(S/((I(P_n^k),A_{n-k+(i-1)}):x_{n-k+i}))=\lceil\frac{n-2k-1+i}{2k+1}\rceil+1\geq \\ \lceil\frac{n-2k-2+i}{2k+1}\rceil+1=\depth(S/((I(P_n^k),A_{n-k+(i-2)}):x_{n-k+(i-1)})).
\end{multline*}
Thus by Lemma \ref{led} we have $$\depth(S/I(P_n^k))=\lceil\frac{n-2k-1}{2k+1}\rceil+1=\lceil\frac{n}{2k+1}\rceil.$$
\end{description}
\end{description}
\end{proof}

\begin{Lemma} \cite[Lemma 4]{AS}\label{PrPFY}
Let $n\geq 2$, then $\sdepth(S/I(P_n^1))=\lceil\frac{n}{3}\rceil$.
\end{Lemma}
\begin{Example}\label{Ex1}
Let $n\geq 2$, and $n\leq 2k+1$, then $\sdepth(S/I(P_n^k))=1$.
\end{Example}
\begin{proof}
If $n\leq k+1$, then by \cite[Theorem 1.1]{MC8} $\sdepth(S/I(P_{n}^k))=1$. Now let $k+2\leq n\leq 2k+1$, then $\depth(S/I(P_n^k))\geq 1$ as $\mathfrak{m}\notin \Ass(S/I(P_n^k))$, thus by \cite[Theorem 1.4]{MC5} $\sdepth(S/I(P_n^k))\geq 1$. Since $x_{k+1}\notin I(P_n^k)$ and $x_ix_{k+1}\in \mathcal{G}(I(P_n^k))$ for all $i\in \{1,\dots k,k+2,\dots,n\}$, therefore $$(I(P_n^k):x_{k+1})=(x_1,\dots x_k,x_{k+2},\dots,x_n).$$
Thus by \cite[Proposition 2.7]{MC} $$\sdepth(S/I(P_n^k))\leq \sdepth(S/(I(P_n^k):x_{k+1}))=\sdepth(S/(x_1,\dots x_k,x_{k+2},\dots,x_n))=1.$$
\end{proof}
\begin{Proposition}\label{Pr11}
Let $k\geq 2$ and $n\geq 2k+2$, then $\sdepth(S/I(P_n^k))\geq\lceil\frac{n}{2k+1}\rceil$.
\end{Proposition}
\begin{proof}
\begin{description}
\item[(1)] If $2k+2\leq n\leq 3k+1$, then we apply Lemma \ref{le02} on the exact sequences in case(b)(3) of Theorem \ref{Th1} and we get $\sdepth(S/I(P_n^k))\geq 2=\lceil\frac{n}{2k+1}\rceil$.
\item[(2)] If $n\geq 3k+2$, then the proof is similar to Theorem \ref{Th1}. We apply Lemma \ref{le02} on the exact sequences in case(b)(4) of Theorem \ref{Th1} we get
\begin{multline*}
\sdepth(S/I(P_n^k))\geq\min\big\{\sdepth(S/(I(P_n^k):x_{n-k})),\sdepth(S/(I(P_n^k),A_{n-1})),\\
\min_{i=1}^{k-1}\{\sdepth(S/((I(P_n^k),A_{n-k+(i-1)}):x_{n-k+i}))\}\big\}\geq\lceil\frac{n}{2k+1}\rceil.
\end{multline*}
\end{description}
\end{proof}
\begin{Theorem}\label{Th11}
Let $n\geq 2$, then $\sdepth(S/I(P_{n}^k))=\lceil\frac{n}{2k+1}\rceil.$
\end{Theorem}
\begin{proof}
By Lemma \ref{PrPFY}, Example \ref{Ex1} and Proposition \ref{Pr11}, it is enough to prove, if $k\geq 2$ then $$\sdepth(S/I(P_n^k))\leq \lceil\frac{n}{2k+1}\rceil, \text{ for } n\geq2k+2.$$ We consider the following three cases:
\begin{description}
\item [(1)] Let $n=(2k+1)l$, where $l\geq 1$. We see that $$v=x_{k+1}x_{3k+2}x_{5k+3}\cdots x_{(2k+1)l-k}\in S\backslash I(P_n^k),$$ but $x_{t_{1}} v\in I(P_{n}^k)$, for all $t_{1}\in[n]\backslash\supp(v)$, by Lemma \ref{zia}, $\sdepth(S/I(P_{n}^k))\leq l= \lceil\frac{n}{2k+1}\rceil$.\\
\item [(2)] Let $n=(2k+1)l+r$, where $r\in\{1,2,3,\dots,k+1\}$ and $l\geq 1$. We have  $$v=x_{k+1}x_{3k+2}x_{5k+3}\cdots x_{(2k+1)l-k}x_{(2k+1)l+r}\in S\backslash I(P_n^k),$$ but $x_{t_{2}}v\in I(P_{n}^k)$, for all $t_{2}\in[n]\backslash\supp(v)$, by Lemma \ref{zia}, $$\sdepth(S/I(P_{n}^k))\leq l+1= \lceil\frac{n}{2k+1}\rceil.$$
\item[(3)] Let $n=(2k+1)l+s,$ where $s\in\{k+2,k+3,\dots,2k\}$ and $l\geq 1$. Since $$v=x_{k+1}x_{3k+2}x_{5k+3}\cdots x_{(2k+1)l+k+1}\in S\backslash I(P_n^k),$$ but $x_{t_{3}} v\in I(P_{n}^k)$, for all $t_{3}\in[n]\backslash\supp(v)$, by Lemma \ref{zia}, $$\sdepth(S/I(P_{n}^k))\leq l+1= \lceil\frac{n}{2k+1}\rceil.$$
\end{description}
\end{proof}

\section{dept and Stanley depth of cyclic modules associated to the edge ideals of the powers of a cycle}




In this section, we compute bounds for depth and Stanley depth for cyclic modules associated to the edge ideals of powers of a cycle.
\begin{Lemma}\label{le111}
Let $k\geq 2$ and $n\geq 3k+2$, then $S/(I(C_n^k),A_{n-1})\cong S_{n-k}/I(P_{n-k}^{k})$.
\end{Lemma}
\begin{proof}
We have \\
$$\mathcal{G}(I(C_n^k))=\mathcal{G}(I(P_n^k))\bigcup\bigcup_{l=1}^{k-1}\{x_l x_{l+n-k},x_l x_{l+n-k+1},\dots,x_l  x_{n-1}\}\bigcup\{x_1x_{n},x_2x_n,\dots,x_{k}x_n\}.$$
Then
$$I(C_n^k)+A_{n-1}=I(P_n^k)+\sum_{l=1}^{k-1}(x_l x_{l+n-k},x_l x_{l+n-k+1},\dots,x_l x_{n-1})+(x_1x_{n},x_2x_n,\dots,x_{k}x_n)+A_{n-1}.$$
Thus by the proof of Lemma \ref{le1}, we have
$$I(P_n^k)+A_{n-1}=I(P_{n-k-1}^{k})+A_{n-1}.$$ And $$\sum_{l=1}^{k-1}(x_l x_{l+n-k},x_l x_{l+n-k+1},\dots,x_l x_{n-1})+A_{n-1}=A_{n-1}.$$ Therefore
\begin{multline*}
$$S/(I(C_n^k),A_{n-1})=S/(I(P_{n-k-1}^{k}),A_{n-1},(x_1x_{n},x_2x_n,\dots,x_{k}x_n))\\\cong K[x_1,x_2,\dots,x_{n-k-1},x_n]/\big(I(P_{n-k-1}^{k}),(x_1x_{n},x_2x_n,\dots,x_{k}x_n)\big).$$
\end{multline*}
After renumbering the variables, we have
$$K[x_1,x_2,\dots,x_{n-k-1},x_n]/\big(I(P_{n-k-1}^{k}),(x_1x_{n},x_2x_n,\dots,x_{k}x_n)\big)\cong S_{n-k}/I(P_{n-k}^k).$$
\end{proof}

\begin{Lemma}\label{le3}
Let $k\geq2$ and $n\geq 3k+2$ and $0\leq i\leq k-1$, then $$S/(I(C_n^k):x_{n-k+i})\cong S_{n-2k-1}/I(P_{n-2k-1}^k)[x_{n-k+i}].$$
\end{Lemma}
\begin{proof}
Let $w$ be a monomial generator of $(I(C_n^k):x_{n-k+i})$, then $w=\frac{v}{gcd(v,x_{n-k+i})}$, where $v\in\mathcal{G}(I(C_n^k))$. If $\supp(v)\cap \mathcal{G}(D_{n-k+i})\neq\emptyset$, then we have $w\in \mathcal{G}(D_{n-k+i})$ and if $\supp(v)\cap \mathcal{G}(D_{n-k+i})=\emptyset$ then $$w\in E:=\mathcal{G}(I(C_{n}^k))\cap K[x_{i+1},x_{i+2},\dots,x_{n-2k-1+i}].$$
So we get $$(I(C_n^k):x_{n-k+i})\subset E+D_{n-k+i},$$ the other inclusion being trivial we get that $(I(C_n^k):x_{n-k+i})=E+D_{n-k+i}.$
Thus we have $S/(I(C_n^k):x_{n-k+i})=S/(E+D_{n-k+i}).$
After renumbering the variables, we have $$S/(I(C_n^k):x_{n-k+i})=S/(E,D_{n-k+i})\cong S_{n-2k-1}/I(P_{n-2k-1}^k)[x_{n-k+i}].$$
\end{proof}
\begin{Lemma}\label{le1111}
Let $k\geq 2$, $n\geq 3k+2$ and $0\leq i\leq k-1$, then we have $$S/\big((I(C_{n}^k),A_{n-k+(i-1)}):x_{n-k+i}\big)\cong S_{n-2k-1}/I(P_{n-2k-1}^k)[x_{n-k+i}].$$
\end{Lemma}
\begin{proof}
Clearly $\big((I(C_{n}^k),A_{n-k+(i-1)}):x_{n-k+i}\big)=\big((I(C_{n}^k):x_{n-k+i}),A_{n-k+(i-1)}\big)$. By using the same arguments as in the proof of Lemma \ref{le3} we have
$$\big((I(C_{n}^k):x_{n-k+i}),A_{n-k+(i-1)}\big)=\big(E,D_{n-k+i},A_{n-k+(i-1)}\big)=
\big(E,D_{n-k+i}\big)$$
as $A_{n-k+(i-1)}\subset D_{n-k+i}$. Thus the required result follows by Lemma \ref{le3}.
\end{proof}

\begin{Theorem}\label{Th22}
Let $n\geq 3$, then
\begin{eqnarray*}
  \depth(S/I(C_n^k)) &=& 1, \text{\,\,\,\,\,\,\,if\,\, $n\leq 2k+1$;}\\
  \depth(S/I(C_n^k)) &\geq& \lceil\frac{n-k}{2k+1}\rceil, \text{\,\,\,\,\,\,\,if\,\, $n\geq 2k+2$.}
\end{eqnarray*}

\end{Theorem}
\begin{proof}
\begin{description}
\item[(a)] Let $n\leq 2k+1$, then by \cite[Corollary 10.3.7]{HH}, $\depth(S/I(C_n^k))=1$.
\item[(b)] Let $n\geq 2k+2$, we consider the following cases:
\begin{description}
\item[(1)] If $k=1$, then by \cite[Proposition 1.3]{MC4} $\depth(S/I(C_n^1))=\lceil\frac{n-1}{3}\rceil$.
\item[(2)] If $k\geq 2$ and $2k+2\leq n\leq 3k+1$, then we have $\depth(S/I(C_n^k))\geq 1=\lceil\frac{n-k}{2k+1}\rceil$ as $\mathfrak{m}\notin \Ass(S/I(C_n^k))$.
\end{description}
\item[(3)] If $k\geq 2$ and $n\geq 3k+2$, consider the family of short exact sequences:
  $$0\longrightarrow S/(I(C_n^k):x_{n-k})=S/((I(C_n^k),A_{n-k-1}):x_{n-k})\xrightarrow[]{\cdot x_{n-k}} S/I(C_n^k)\longrightarrow S/(I(C_n^k),A_{n-k})\longrightarrow 0$$
  $$0\longrightarrow S/((I(C_n^k),A_{n-k}):x_{n-k+1})\xrightarrow[]{\cdot x_{n-k+1}} S/(I(C_n^k),A_{n-k})\longrightarrow S/(I(C_n^k),A_{n-k+1})\longrightarrow 0$$
  $$\vdots$$
   $$0\longrightarrow S/((I(C_n^k),A_{n-k+(i-1)}):x_{n-k+i})\xrightarrow[]{\cdot x_{n-k+i}} S/(I(C_n^k),A_{n-k+(i-1)})\longrightarrow S/(I(C_n^k),A_{n-k+i})\longrightarrow 0$$
   $$\vdots$$
   $$0\longrightarrow S/((I(C_n^k),A_{n-2}):x_{n-1})\xrightarrow[]{\cdot x_{n-1}} S/(I(C_n^k),A_{n-2})\longrightarrow S/(I(C_n^k),A_{n-1})\longrightarrow 0$$
By Lemma \ref{le111} we have
$$S/(I(C_n^k),A_{n-1}))\cong S_{n-k}/I(P_{n-k}^k).$$
Let $0\leq i\leq k-1$, then by Lemma \ref{le1111}, we have
$$S/((I(C_n^k),A_{n-k+(i-1)}):x_{n-k+i})\cong S_{n-2k-1}/I(P_{n-2k-1}^k)[x_{n-k+i}].$$
By Theorem \ref{Th1} and \cite[Lemma 3.6]{HVZ}, we have $$\depth(S/((I(C_n^k),A_{n-k+(i-1)}):x_{n-k+i}))=\lceil\frac{n-2k-1}{2k+1}\rceil+1=\lceil\frac{n}{2k+1}\rceil.$$
 Again by Theorem \ref{Th1}, we have
 $$\depth (S/(I(C_n^k),A_{n-1}))=\lceil\frac{n-k}{2k+1}\rceil.$$
Thus by Depth Lemma $\depth(S/I(C_n^k))\geq \lceil\frac{n-k}{2k+1}\rceil.$
\end{description}
\end{proof}
\begin{Corollary}\label{Cor11}
Let $n\geq 3$, if $n\leq 2k+1$, then $\depth(S/I(C_n^k))=1.$ If $n\geq 2k+2$, then
\begin{eqnarray*}
&\depth(S/I(C_{n}^k))=\lceil\frac{n}{2k+1}\rceil,& \text{\,\,\,if\,\, $n\equiv 0,k+1,\dots,n-1\,(\mod (2k+1))$};\\
&\lceil\frac{n}{2k+1}\rceil-1\leq\depth(S/I(C_{n}^k))\leq \lceil\frac{n}{2k+1}\rceil,&  \text{\,\,\,if\,\, $n\equiv 1,\dots,k\,(\mod (2k+1))$}.
\end{eqnarray*}
\end{Corollary}
\begin{proof}
By Theorem \ref{Th22}, it is enough to prove that $\depth(S/I(C_n^k))\leq \lceil\frac{n}{2k+1}\rceil$, for $k\geq 2$ and $n\geq 2k+2$. Since $x_{n-k}\notin I(C_n^k)$, thus by \cite[Corollary 1.3]{AR1} we have $$\depth(S/I(C_n^k))\leq\depth(S/(I(C_n^k):x_{n-k})).$$ Now we consider two cases:
\begin{description}
\item [(1)] Let $2k+2\leq n\leq 3k+1$, then $S/(I(C_n^k):x_{n-k})=S/(I(P_n^k):x_{n-k})$ so by the proof of Theorem \ref{Th1} we have $\depth(S/(I(P_n^k):x_{n-k}))=2=\lceil\frac{n}{2k+1}\rceil$. Therefore $$\depth(S/I(C_n^k))\leq\depth(S/(I(C_n^k):x_{n-k}))=2=\lceil\frac{n}{2k+1}\rceil.$$
\item [(2)] Let $n\geq 3k+2$, then by Lemma \ref{le3}, $S/(I(C_n^k):x_{n-k})\cong S_{n-2k-1}/I(P_{n-2k-1}^k)[x_{n-k}]$. By \cite[Lemma 3.6]{HVZ} and Theorem \ref{Th1}, $\depth (S_{n-2k-1}/I(P_{n-2k-1}^k)[x_{n-k}])=\lceil\frac{n}{2k+1}\rceil$. Thus
$$\depth(S/I(C_n^k))\leq \depth(S/(I(C_n^k):x_{n-k}))=\lceil\frac{n}{2k+1}\rceil.$$
\end{description}

\end{proof}

\begin{Theorem}\label{Th33}
Let $n\geq 3$, then
\begin{eqnarray*}
  \sdepth(S/I(C_n^k))&=&1, \text{\,\,\,\,\,\,\,if\,\, $n\leq 2k+1$;}\\
  \sdepth(S/I(C_n^k))&\geq&\lceil\frac{n-k}{2k+1}\rceil, \text{\,\,\,\,\,\,\,if\,\, $n\geq 2k+2$.}
\end{eqnarray*}
\end{Theorem}
\begin{proof}
\begin{description}
\item[(a)] Let $n\leq 2k+1$, then $\sdepth(S/I(C_n^k))=1$ by \cite[Theorem 1.1]{MC8}.
\item[(b)] Let $n\geq 2k+2$, consider the following cases:
\begin{description}
\item[(1)] If $k=1$, then by \cite[Proposition 1.8]{MC4} $\sdepth(S/I(C_n^1))\geq\lceil\frac{n-1}{3}\rceil$.
\item[(2)] If $k\geq 2$, $2k+2\leq n\leq 3k+1$, then $\depth(S/I(C_n^k))\geq 1$ as $\mathfrak{m}\notin \Ass(S/I(C_n^k))$, thus by \cite[Theorem 1.4]{MC5}, $\sdepth(S/I(C_n^k))\geq 1=\lceil\frac{n-k}{2k+1}\rceil$.
\item[(3)] If $k\geq 2$, $n\geq 3k+2$, this theorem can be proved by using similar arguments of Theorem \ref{Th22}. We apply Lemma \ref{le02} on the exact sequences of Theorem \ref{Th22} we get
\begin{multline*}
\sdepth(S/I(C_n^k))\geq\min\big\{\sdepth(S/(I(C_n^k):x_{n-k})),\sdepth(S/(I(C_n^k),A_{n-1})),\\
\min_{i=1}^{k-1}\{\sdepth(S/((I(C_n^k),A_{n-k+(i-1)}):x_{n-k+i}))\}\big\}\geq\lceil\frac{n-k}{2k+1}\rceil.
\end{multline*}
\end{description}
\end{description}
\end{proof}
\begin{Corollary}\label{Cor22}
Let $n\geq 3$, if $n\leq 2k+1$, then $\sdepth(S/I(C_n^k))=1.$ If $n\geq 2k+2$, then
\begin{eqnarray*}
&\sdepth(S/I(C_{n}^k))=\lceil\frac{n}{2k+1}\rceil,& \text{\,\,\,if\,\, $n\equiv 0,k+1,\dots,n-1\,(\mod (2k+1))$};\\
&\lceil\frac{n}{2k+1}\rceil-1\leq\sdepth(S/I(C_{n}^k))\leq \lceil\frac{n}{2k+1}\rceil,&  \text{\,\,\,if\,\, $n\equiv 1,\dots,k\,(\mod (2k+1))$}.
\end{eqnarray*}

\end{Corollary}
\begin{proof}
When $k=1$, then by \cite[Theorem 1.9]{MC4}, $\sdepth(S/I(C_{n}^k))\leq \lceil\frac{n}{3}\rceil$. By Theorem \ref{Th33} it is enough to prove that $\sdepth(S/I(C_{n}^k))\leq \lceil\frac{n}{2k+1}\rceil$, for $k\geq 2$ and $n\geq 2k+2$. The proof of Corollary \ref{Cor11} also works for Stanley depth by using \cite[Proposition 2.7]{MC} instead of \cite[Corollary 1.3]{AR1} and Theorem \ref{Th11} instead of Theorem \ref{Th1}.
\end{proof}
\section{Lower bounds for Stanley depth of edge ideals of $k^{th}$ powers of paths and cycles and a conjecture of Herzog }
In this section we compute some lower bounds for Stanley depth of $I(P_n^k)$ and $I(C_n^k)$. These bounds are good enough to prove that Conjecture \ref{C1} is true for $I(P_n^k)$ and $I(C_n^k)$. Let $0\leq i\leq k-1$, define $$R_{n-k+i}:=K[\{x_1,x_2,\ldots,x_n\}\backslash \{x_{n-k},x_{n-k+1},\ldots,x_{n-k+i}\}]$$ and $$B'_{n-k+i}:=(N_{P_n^k}(x_{n-k+i})\backslash \{x_{n-k},x_{n-k+1},\ldots,x_{n-k+(i-1)}\}),$$
where $R_{n-k+i}$ is a subring of $S$ and $B'_{n-k+i}$ is a monomial prime ideal of $S$ generated by $$N_{P_n^k}(x_{n-k+i})\backslash \{x_{n-k},x_{n-k+1},\ldots,x_{n-k+(i-1)}\}.$$
Let $I\subset Z=K[x_{i_1},x_{i_1},\dots,x_{i_r}]$ be a monomial ideal and let $Z':=Z[x_{i_r+1}]$, then we write $IZ'=I[x_{i_r+1}]$. We recall in the following a useful remark of Cimpoeas.
\begin{Remark}\cite[Remark 1.7]{MC}\label{reI1}
{\em
Let $I$ be a monomial ideal of $S$, and $I'=(I,x_{n+1},x_{n+2},\dots,x_{n+m})$ be a monomial ideal of $S'=S[x_{n+1},x_{n+2},\dots,x_{n+m}]$, then
$$\sdepth_{S'}(I')\geq \min\{\sdepth_S(I)+m,\sdepth_S(S/I)+\lceil\frac{m}{2}\rceil\}.$$
}
\end{Remark}
\begin{Theorem}\label{th3}
Let $n\geq 2$, then $\sdepth (I(P_n^k))\geq \lceil\frac{n}{2k+1}\rceil+1$.
\end{Theorem}
\begin{proof}
\begin{description}
\item [(a)] Let $n\leq 2k+1$, since the minimal generators of $I(P_n^k)$ have degree 2, so by \cite[Lemma 2.1]{KS} $\sdepth (I(P_n^k))\geq2=\lceil\frac{n}{2k+1}\rceil+1$ .\\
\item [(b)] Let $n\geq 2k+2$, if $k=1$, then by \cite[Theorem 2.3]{O} we have $\sdepth(I(P_n^1))\geq n-\lfloor\frac{n-1}{2}\rfloor=\lceil\frac{n-1}{2}\rceil+1\geq \lceil\frac{n}{3}\rceil+1$. Now let $k\geq 2$, we prove this result by induction on $n$. We consider the following decomposition of $I(P_n^k)$ as a vector space: \\
$$I(P_n^k)=I(P_n^k)\cap R_{n-k}\bigoplus x_{n-k}(I(P_n^k):x_{n-k})S.$$ Similarly, we can decompose $I(P_n^k)\cap R_{n-k}$ by the following:
$$I(P_n^k)\cap R_{n-k}=I(P_n^k)\cap R_{n-k+1}\bigoplus x_{n-k+1}(I(P_n^k)\cap R_{n-k}:x_{n-k+1})R_{n-k}.$$
Continuing in the same way for $1\leq i\leq k-1$ we have
$$I(P_n^k)\cap R_{n-k+i}=I(P_n^k)\cap R_{n-k+(i+1)}\bigoplus x_{n-k+(i+1)}(I(P_n^k)\cap R_{n-k+i}:x_{n-k+(i+1)})R_{n-k+i}.$$
Finally we get the following decomposition of $I(P_n^k)$:
\begin{multline*}
I(P_n^k)=I(P_n^k)\cap R_{n-1}\bigoplus\bigoplus_{i=1}^{k-1}x_{n-k+i}(I(P_n^k)\cap R_{n-k+(i-1)}:x_{n-k+i})R_{n-k+i}\\\bigoplus x_{n-k}(I(P_n^k):x_{n-k})S.
\end{multline*}
Therefore
\begin{multline}\label{eq3}
\sdepth(I(P_n^k))\geq \min\Big\{\sdepth(I(P_n^k)\cap R_{n-1}),\sdepth((I(P_n^k):x_{n-k})S),\\\min_{i=1}^{k-1}\{\sdepth((I(P_n^k)\cap R_{n-k+(i-1)}:x_{n-k+i})R_{n-k+i})\}\Big\}
\end{multline}
Now $$I(P_n^k)\cap R_{n-1}=\mathcal{G}(I(P_{n-k-1}^k))[x_n],$$
thus by induction on $n$ and \cite[Lemma 3.6]{HVZ} we have $$\sdepth(I(P_n^k)\cap R_{n-1})\geq \lceil\frac{n-k-1}{2k+1}\rceil+1+1\geq \lceil\frac{n}{2k+1}\rceil+1.$$
Now we need to show that $$\sdepth((I(P_n^k):x_{n-k})S)\geq \lceil\frac{n}{2k+1}\rceil+1$$ and $$\sdepth((I(P_n^k)\cap R_{n-k+(i-1)}:x_{n-k+i})R_{n-k+i})\geq \lceil\frac{n}{2k+1}\rceil+1.$$
For this we consider the following cases:
\begin{description}
\item [(1)] Let $2k+2\leq n\leq 3k+1$.\\
If $n=2k+2$, then $(I(P_n^k):x_{n-k})S=(x_2,\dots,x_{n-k-1},x_{n-k+1},\dots,x_n)S,$ thus by \cite[Theorem 2.2]{BH} and \cite[Lemma 3.6]{HVZ} we have $$\sdepth((I(P_n^k):x_{n-k})S)=\lceil\frac{n-2}{2}\rceil+2\geq \lceil\frac{n}{2k+1}\rceil+1.$$
If $2k+3\leq n\leq 3k+1$, then by Remark \ref{re1} $(I(P_n^k):x_{n-k})S=(\mathcal{G}(I(P_{n-2k-1}^{f(n-k)})),B_{n-k})[x_{n-k}].$
Since $\sdepth(I(P_{n-2k-1}^{f(n-k)}))+|\mathcal{G}(B_{n-k})|\geq 2$, by Remark \ref{rev} we have $$\sdepth(S_{n-2k-1}/I(P_{n-2k-1}^{f(n-k)}))+\lceil\frac{|\mathcal{G}(B_{n-k})|}{2}\rceil\geq 2,$$
then by Remark \ref{reI1},
$$\sdepth(\mathcal{G}(I(P_{n-2k-1}^{f(n-k)})),B_{n-k})\geq 2,$$
and by \cite[Lemma 3.6]{HVZ} we have
$$\sdepth((I(P_n^k):x_{n-k})S)\geq 3=\lceil\frac{n}{2k+1}\rceil+1.$$ Now since
$$(I(P_n^k)\cap R_{n-k+(i-1)}:x_{n-k+i})R_{n-k+i})=(\mathcal{G}(I(P_{n-2k-1+i}^{f(n-k+i)})),B'_{n-k+i})[x_{n-k+i}].$$
So by the same arguments we have $$\sdepth((I(P_n^k)\cap R_{n-k+(i-1)}:x_{n-k+i})R_{n-k+i})\geq 3=\lceil\frac{n}{2k+1}\rceil+1.$$
\item[(2)] Let $n\geq 3k+2$, then by the proof of Lemma \ref{le2}
$$(I(P_n^k):x_{n-k})S=(\mathcal{G}(I(P_{n-2k-1}^k)),B_{n-k})[x_{n-k}]$$ and
$$(I(P_n^k)\cap R_{n-k+(i-1)}:x_{n-k+i})R_{n-k+i}=(\mathcal{G}(I(P_{n-2k-1+i}^k)),B'_{n-k+i})[x_{n-k+i}].$$
By Remark \ref{reI1} we have
\begin{multline*}
\sdepth(\mathcal{G}(I(P_{n-2k-1}^k)),B_{n-k})\geq \\ \min\Big\{\sdepth(\mathcal{G}(I(P_{n-2k-1}^k)))+|\mathcal{G}(B_{n-k})|,\sdepth(S_{n-2k-1}/I(P_{n-2k-1}^k))+\lceil\frac{|\mathcal{G}(B_{n-k})|}{2}\rceil\Big\}.
\end{multline*}
By induction on $n$ we have $\sdepth(\mathcal{G}(I(P_{n-2k-1}^k)))\geq \lceil\frac{n-2k-1}{2k+1}\rceil+1=\lceil\frac{n}{2k+1}\rceil$, and by Theorem \ref{Th11} we have $\sdepth(S_{n-2k-1}/I(P_{n-2k-1}^k))=\lceil\frac{n}{2k+1}\rceil-1$. Therefore $$\sdepth(\mathcal{G}(I(P_{n-2k-1}^k)),B_{n-k})\geq \lceil\frac{n}{2k+1}\rceil+1.$$
Thus by \cite[Lemma 3.6]{HVZ} we have $\sdepth((I(P_n^k):x_{n-k})S)>\lceil\frac{n}{2k+1}\rceil+1.$\\
Now using Remark \ref{reI1} again, we get
\begin{multline*}
\sdepth(\mathcal{G}(I(P_{n-2k-1+i}^k)),B'_{n-k+i})\geq \\ \min\Big\{\sdepth(\mathcal{G}(I(P_{n-2k-1+i}^k)))+|\mathcal{G}(B'_{n-k+i})|,\sdepth(S_{n-2k-1+i}/I(P_{n-2k-1+i}^k))+\lceil\frac{|\mathcal{G}(B'_{n-k+i})|}{2}\rceil\Big\}.
\end{multline*}
By induction on $n$ we have $\sdepth(\mathcal{G}(I(P_{n-2k-1+i}^k)))\geq \lceil\frac{n-2k-1+i}{2k+1}\rceil+1$, and by Theorem \ref{Th11} we have $\sdepth(S_{n-2k-1+i}/I(P_{n-2k-1+i}^k))=\lceil\frac{n-2k-1+i}{2k+1}\rceil$. Therefore $$\sdepth(\mathcal{G}(I(P_{n-2k-1+i}^k)),B'_{n-k+i})\geq \lceil\frac{n-2k-1+i}{2k+1}\rceil+1.$$
Thus by \cite[Lemma 3.6]{HVZ}, $\sdepth((I(P_n^k)\cap R_{n-k+(i-1)}:x_{n-k+i})R_{n-k+i})\geq \lceil\frac{n}{2k+1}\rceil+1$.
This completes the proof.
\end{description}
\end{description}
\end{proof}

\begin{Proposition}\label{pr3}
 Let $n\geq 2k+1$, then $\sdepth({I(C_{n}^k)}/{I(P_{n}^k)})\geq\lceil\frac{n+k+1}{2k+1}\rceil.$
\end{Proposition}
\begin{proof}
When $k=1$, then by \cite[Proposition 1.10]{MC4} we have the required result. Now assume that $k\geq 2$ and consider the following cases:
\begin{description}
\item[(1)] If $2k+1\leq n\leq 3k+1.$ Since $I(C_{n}^k)$ is a monomial ideal generated by degree $2$ so by \cite[Theorem 2.1]{HVZ} $\sdepth(I(C_{n}^k)/I(P_{n}^k))\geq 2=\lceil\frac{n+k+1}{2k+1}\rceil.$

\item[(2)] When $3k+2\leq n\leq 4k+1$, then we use \cite{HVZ} to show that there exist Stanley decompositions of desired Stanley depth. Let $s\in \{1,2,\dots k\}$, $j_s\in\{1,2,\dots,k+1-s\}$ and
$$L:=\bigoplus^{k}_{s=1}\Big(\bigoplus^{k+1-s}_{j_{s}=1}x_{j_{s}}x_{n+1-s}K[x_{j_{s}},x_{j_{s}+k+1},x_{n+1-s}]\Big).$$
It is easy to see that $L\subset I(C_{n}^k)\backslash I(P_{n}^k)$. Now let $u_{i}\in I(C_{n}^k)\backslash I(P_{n}^k)$ be a squarefree monomial such that $u_i\notin L$ then clearly $\deg (u_{i})\geq 3$. Since
$$I(C_{n}^k)/I(P_{n}^k)\cong L\bigoplus_{u_{i}}u_iK[\,\supp (u_{i})]$$
Thus $\sdepth({I(C_{n}^k)}/{I(P_{n}^k)})\geq3=\lceil\frac{n+k+1}{2k+1}\rceil$ as required.

\item[(3)] Now, assume that $n\geq4k+2$. We have the following $K$-vector space isomorphism:
$$I(C_{n}^k)/I(P_{n}^k)\cong \bigoplus^{k}_{j_{1}=1} x_{j_{1}}x_{n}\frac{K[x_{j_{1}+k+1},x_{j_{1}+k+2},\dots,x_{n-k-1}]}{(x_{j_{1}+k+1}x_{j_{1}+k+2},x_{j_{1}+k+1}x_{j_{1}+k+3},\dots,x_{n-k-2}x_{n-k-1})}[x_{j_{1}},x_{n}] $$
$$\bigoplus\bigoplus^{k-1}_{j_{2}=1} x_{j_{2}}x_{n-1}\frac{K[x_{j_{2}+k+1},x_{j_{2}+k+2},\dots,x_{n-k-2}]}{(x_{j_{2}+k+1}x_{j_{2}+k+2},x_{j_{2}+k+1}x_{j_{1}+k+3},\dots,x_{n-k-3}x_{n-k-2})}[x_{j_{2}},x_{n-1}] \bigoplus$$
$$\dots$$
$$\bigoplus^{2}_{j_{k-1}=1} x_{j_{k-1}}x_{n-(k-2)}\frac{K[x_{j_{k-1}+k+1},x_{j_{k-1}+k+2},\dots,x_{n-2k+1}]}{(x_{j_{k-1}+k+1}x_{j_{k-1}+k+2},x_{j_{k-1}+k+1}x_{j_{k-1}+k+3},\dots,x_{n-2k}x_{n-2k+1})}[x_{j_{k-1}},x_{n-k+2}] $$
$$\bigoplus x_{1}x_{n-(k-1)}\frac{K[x_{k+2},x_{k+3},\dots,x_{n-2k}]}{(x_{k+2}x_{k+3},x_{k+2}x_{k+4},\dots,x_{n-2k-1}x_{n-2k})}[x_{1},x_{n-(k-1)}]. $$
Thus
$$I(C_{n}^k)/I(P_{n}^k)\cong\bigoplus^{k}_{s=1}\Big(\bigoplus^{k+1-s}_{j_{s}=1} x_{j_{s}}x_{n+1-s}\big({{S_{{j_{s}+k+1},{n-s-k}}}}/\big(\mathcal{G}(I(P_{n}^k))\cap S_{{j_{s}+k+1},{n-s-k}} \big)[x_{j_{s}},x_{n+1-s}]\Big),$$
where $S_{j_{s}+k+1,{n-s-k}}=K[x_{j_{s}+k+1},x_{j_{s}+k+2},\dots,x_{n-s-k}].$
Indeed, if $u\in I(C_{n}^k)$ such that $u\not\in I(P_{n}^k)$ then $(x_{j_{s}}x_{n+1-s})|u,$ for only one pair of $s$ and $j_s$. If $(x_{j_{s}}x_{n+1-s})|u$, then $u=x^\gamma_{j_s}x^\delta_{n+1-s}v$, $v\in S_{j_{s}+k+1,{n-s-k}},$ since $v\notin I(P_{n}^k)$, it follows that $v\notin\mathcal{G}(I(P_{n}^k))\cap S_{j_{s}+k+1,{n-s-k}}.$ Clearly, $$S_{j_{s}+k+1,{n-s-k}}/\mathcal{G}(I(P_{n}^k))\cap S_{j_{s}+k+1,{n-s-k}}\cong S_{n-(j_s+2k+s)}/I(P_{n-(j_s+2k+s)}^k).$$
Thus by Theorem \ref{Th11} and \cite[Lemma 3.6]{HVZ}, we have

$$\sdepth(I(C_{n}^k)/I(P_{n}^k))\geq\min_{s=1}^{k}\{\lceil\frac{n-(j_{s}+s+2k)}{2k+1}\rceil+2\}.$$
It is easy to see that $\max\{j_{s}+s\}=k+1.$ Therefore
$$\sdepth(I(C_{n}^k)/I(P_{n}^k))\geq\lceil\frac{n-(3k+1)}{2k+1}\rceil+2=\lceil\frac{n+k+1}{2k+1}\rceil.$$
\end{description}
\end{proof}

\begin{Theorem}\label{th5}
Let $n\geq 3$, then
\begin{eqnarray*}
  \sdepth(I(C_n^k))&\geq& 2, \text{\,\,\,\,\,\,if\,\, $n\leq 2k+1$}; \\
  \sdepth(I(C_n^k))&\geq&  \lceil\frac{n-k}{2k+1}\rceil+1, \text{\,\,\,\,\,\,if\,\, $n\geq 2k+2$}.
\end{eqnarray*}
\end{Theorem}
\begin{proof}
\begin{description}
\item [(a)] Let $n\leq 2k+1$, since the minimal generators of $I(C_n^k)$ have degree 2, so by \cite[Lemma 2.1]{KS} $\sdepth (I(C_n^k))\geq2$.
\item [(b)] Let $n\geq 2k+2$, consider the short exact sequence $$0\longrightarrow I(P_n^k)\longrightarrow I(C_n^k)\longrightarrow I(C_n^k)/I(P_n^k)\longrightarrow 0,$$
then by Lemma \ref{le02}, $$\sdepth (I(C_n^k))\geq \min\{\sdepth (I(P_n^k)),\sdepth(I(C_n^k)/I(P_n^k))\}.$$
By Theorem \ref{th3}, $$\sdepth (I(P_n^k))\geq \lceil\frac{n}{2k+1}\rceil+1,$$ and by Proposition \ref{pr3}, we have
$$\sdepth(I(C_n^k)/I(P_n^k))\geq \lceil\frac{n+k+1}{2k+1}\rceil=\lceil\frac{n-k}{2k+1}\rceil+1,$$
this finishes the proof.
\end{description}
\end{proof}
\begin{Corollary}\label{Cor4}
Let $n\geq 3$, if $n\leq 2k+1$, then $$\sdepth(I(C_{n}^k))\geq 2=\sdepth(S/I(C_{n}^k))+1.$$ If $n\geq 2k+2$, then
\begin{eqnarray*}
  \sdepth(I(C_{n}^k)) &\geq& \sdepth(S/I(C_{n}^k)),  \text{\,\,\,if\,\, $n\equiv 1,\dots,k\,(\mod (2k+1))$}; \\
  \sdepth(I(C_{n}^k)) &\geq& \sdepth(S/I(C_{n}^k))+1,\text{\,\,\,if\,\, $n\equiv 0,k+1,\dots,n-1\,(\mod (2k+1))$}.
\end{eqnarray*}
\end{Corollary}
\begin{proof}
Proof follows by Corollary \ref{Cor22} and Theorem \ref{th5}.
\end{proof}




\end{document}